\documentclass[10pt]{article}

\usepackage[T1]{fontenc}
\usepackage{lmodern}
\usepackage[margin=0.9in]{geometry}
\usepackage{amsmath,amssymb,amsthm,mathtools}
\usepackage{enumitem}
\usepackage[colorlinks=true,linkcolor=blue,citecolor=blue,urlcolor=blue]{hyperref}
\usepackage{pgfplots}
\pgfplotsset{compat=1.18}
\title{Extremal Chowla Sets and Their Linear Analogues: A Human--AI Mathematical Investigation Using Co-Scientist}
\date{}

\newtheorem{theorem}{Theorem}[section]
\newtheorem{proposition}[theorem]{Proposition}
\newtheorem{lemma}[theorem]{Lemma}
\newtheorem{corollary}[theorem]{Corollary}
\newtheorem{question}[theorem]{Question}
\theoremstyle{definition}
\newtheorem{definition}[theorem]{Definition}
\newtheorem{remark}[theorem]{Remark}

\newcommand{\F}{\mathbb{F}}
\newcommand{\Z}{\mathbb{Z}}
\newcommand{\N}{\mathbb{N}}
\newcommand{\Ccal}{\mathcal{C}}
\newcommand{\ord}{\operatorname{ord}}
\newcommand{\Exp}{\operatorname{exp}}
\newcommand{\Gal}{\operatorname{Gal}}
\newcommand{\lcm}{\operatorname{lcm}}
\newcommand{\gbinom}[3]{\genfrac{[}{]}{0pt}{}{#1}{#2}_{#3}}
\newcommand{\CoS}{Co-Scientist}

\usepackage{authblk}
\usepackage{array}
\usepackage{tabularx}
\usepackage{longtable}
\usepackage{booktabs}
\usepackage{xcolor}
\usepackage{tikz}
\usepackage{tcolorbox}
\tcbuselibrary{breakable}
\usetikzlibrary{positioning, arrows.meta}

\author[1]{Mohsen~Aliabadi}
\author[1]{Keith~Driscoll}
\author[1]{Elliot~Krop}
\author[2]{Petar~Sirkovic}
\author[1]{Everett~Sullivan}
\author[3]{Elahe~Vedadi}

\affil[1]{Clayton State University}
\affil[2]{Google Cloud AI Research}
\affil[3]{Google DeepMind}

\begin{document}

\maketitle

\begin{abstract}
We study an extremal invariant arising from Chowla-type order conditions in
finite groups. A nonempty subset \(S\) of a finite group \(G\) is called a
Chowla set if every element of \(S\) has order greater than \(|S|\), and
\(\Ccal(G)\) denotes the largest possible cardinality of such a set. We show
that \(\Ccal(G)\) is determined by the order distribution of \(G\). For cyclic
groups, this yields an exact divisor formula and a criterion for
\(\Ccal(\mathbb Z/n\mathbb Z)=\varphi(n)\). We prove
\[
 \liminf_{n\to\infty}\frac{\Ccal(\mathbb Z/n\mathbb Z)}{\varphi(n)}=1,
 \qquad
 \limsup_{n\to\infty}\frac{\Ccal(\mathbb Z/n\mathbb Z)}{\varphi(n)}=\infty,
\]
and determine the corresponding lower and upper limits after normalization by
\(n\). For finite abelian groups, we obtain an explicit formula from the
invariant-factor decomposition and a closed form for abelian \(p\)-groups.

We also study a linear version for finite field extensions. A nonzero
\(K\)-subspace \(A\subseteq L\) is called a Chowla subspace if
\([K(a):K]>\dim_K A\) for every \(0\ne a\in A\). This dimension-dependent
condition does not, in general, require every nonzero element of \(A\) to
generate \(L\). If \(L/K\) is finite separable, however, we prove the exact
extremal formula
\[
 \Ccal(L/K)=[L:K]-d_{\max}(L/K),
\]
where \(d_{\max}(L/K)\) is the largest degree of a proper intermediate field.
For finite fields, a normal-basis construction gives a direct proof in every
degree. We further obtain explicit counting formulas in the prime-power chain
case and in the first non-chain case of degree \(pq\).

The mathematical results were developed through a human--AI collaboration.
A lightweight, reasoning-oriented configuration of \CoS{} assisted with
candidate reductions, examples, proof strategies, and draft arguments.  The
human authors selected the problem, designed the sequence of mathematical
questions, checked every argument, repaired gaps, supplied missing hypotheses,
and wrote the final proofs.  We document this expert-guided workflow as a case
study in AI-assisted mathematical discovery; the AI systems were not used as
formal proof verifiers.
\end{abstract}



\section{Introduction}

The classical Cauchy--Davenport theorem asserts that if \(p\) is prime and
\(A,B\subseteq \mathbb Z/p\mathbb Z\) are nonempty, then
\[
   |A+B|\ge \min\{p, |A|+|B|-1\},
\]
where \(A+B=\{a+b:a\in A,\ b\in B\}\).  The theorem was proved by Cauchy
and later rediscovered by Davenport \cite{Cauchy1813,Davenport1935}.  Its
prime-modulus hypothesis is essential in this form.  Indeed, when the modulus
is composite, the group \(\mathbb Z/n\mathbb Z\) contains nonzero elements of
proper additive order, and these elements may create periodic obstructions to
Cauchy--Davenport type growth.

Chowla proved a composite-modulus analogue which removes these obstructions by
imposing an order condition on one of the summand sets \cite{Chowla1935}.  Let \(n\ge 2\), and let
\(A,B\subseteq \mathbb Z/n\mathbb Z\) be nonempty.  Suppose that \(0\in B\) and
that
\[
   \gcd(b,n)=1 \qquad \text{for every } b\in B\setminus\{0\}.
\]
Equivalently, every nonzero element of \(B\) has additive order \(n\) in
\(\mathbb Z/n\mathbb Z\).  Then Chowla's theorem gives
\[
   |A+B|\ge \min\{n, |A|+|B|-1\}.
\]
Thus Chowla's theorem extends the Cauchy--Davenport lower bound from prime
moduli to arbitrary moduli, provided the nonzero elements of one summand set
generate the whole cyclic group.

Subsequent work developed group-theoretic versions in which the full-order
condition is replaced by weaker lower bounds on element orders.  In particular,
Hamidoune's isoperimetric method studies additive expansion under hypotheses
requiring the relevant nonidentity elements to have order at least the size of
the set being considered \cite{Hamidoune1996}.  The notion of a Chowla set
considered in this paper is obtained by isolating this order condition from
that broader additive-combinatorial setting.
 In all of these settings, the common role of the
order hypothesis is to rule out small cyclic subgroups and the subgroup-periodic
examples that obstruct Cauchy--Davenport type lower bounds.

Linear analogues of classical addition theorems for products of vector
subspaces were developed by Hou, Leung, and Xiang
\cite{HouLeungXiang2002}. Linear Kneser and Vosper theorems were
subsequently refined by Bachoc, Serra, and Z\'emor
\cite{BachocSerraZemor2018,BachocSerraZemor2017}. These results provide a
natural additive-combinatorial setting for replacing group-theoretic order
conditions by degree conditions in field extensions.

The linear invariant considered here is dimension dependent. If
\(\dim_K A=m\), the Chowla condition excludes only those nonzero elements
whose generated subextensions have degree at most \(m\). Thus a Chowla
subspace may contain elements generating proper intermediate fields of degree
larger than \(m\). At the extremal dimension, the situation becomes more
rigid, and this leads to an exact formula in terms of the largest proper
intermediate field.

\subsection{Main results}
The purpose of this paper is to isolate and study the extremal content of
Chowla-type order conditions.

For a finite group \(G\), let
\[
 N_G(m)=|\{x\in G:\ord(x)>m\}|.
\]
We first prove
\[
 \Ccal(G)=\max\{m\in\mathbb N:N_G(m)\ge m\},
\]
so both existence and enumeration of Chowla sets are determined by the order
distribution of \(G\). For cyclic groups, this becomes
\[
 \Ccal(\mathbb Z/n\mathbb Z)
 =\max\left\{m\in\mathbb N:
 \sum_{\substack{d\mid n\\ d>m}}\varphi(d)\ge m\right\}.
\]
If \(p_1\) is the least prime divisor of \(n\), then
\[
 \Ccal(\mathbb Z/n\mathbb Z)=\varphi(n)
 \quad\Longleftrightarrow\quad
 \varphi(n)\ge \frac{n}{p_1}-1.
\]
We also establish
\[
 \liminf_{n\to\infty}\frac{\Ccal(\mathbb Z/n\mathbb Z)}{\varphi(n)}=1,
 \qquad
 \limsup_{n\to\infty}\frac{\Ccal(\mathbb Z/n\mathbb Z)}{\varphi(n)}=\infty,
\]
and
\[
 \liminf_{n\to\infty}\frac{\Ccal(\mathbb Z/n\mathbb Z)}{n}=0,
 \qquad
 \limsup_{n\to\infty}\frac{\Ccal(\mathbb Z/n\mathbb Z)}{n}=1.
\]
Moreover, every fixed inequality
\(\Ccal(\mathbb Z/n\mathbb Z)\ge M\varphi(n)\) holds on a set of positive
lower asymptotic density.

For a finite abelian group
\(G\cong C_{n_1}\oplus\cdots\oplus C_{n_r}\), with
\(n_1\mid\cdots\mid n_r\), M\"obius inversion gives an explicit formula for
the number of elements of each exact order and hence for \(\Ccal(G)\). In
particular, if \(G\) is an abelian \(p\)-group of exponent \(p^\ell\), then
\[
 \Ccal(G)=
 \begin{cases}
 p^\ell-p^{\ell-1},&\text{if \(G\) is cyclic},\\
 p^\ell-1,&\text{if \(G\) is noncyclic}.
 \end{cases}
\]

For a finite extension \(L/K\), we define Chowla subspaces by the condition
\([K(a):K]>\dim_K A\) for all \(0\ne a\in A\). Writing
\[
 d_{\max}(L/K)=
 \max\{[E:K]:K\subseteq E\subsetneq L\},
\]
we prove the upper bound
\[
 \Ccal(L/K)\le [L:K]-d_{\max}(L/K)
\]
for every finite extension. We then prove the exact formula
\[
 \boxed{\Ccal(L/K)=[L:K]-d_{\max}(L/K)}
\]
for every finite separable extension. The infinite-base-field case follows
from a direct finite-union avoidance construction, while the finite-field case
is proved by an explicit normal-basis construction.

Finally, we count Chowla subspaces over finite fields. When \(n=p^r\), the
intermediate fields form a chain, and the number of \(m\)-dimensional Chowla
subspaces is a single Gaussian-binomial expression. When \(n=pq\) with
\(p<q\) prime, two incomparable proper subfields occur; a subspace-lattice
M\"obius inversion gives an explicit inclusion--exclusion formula.

\section{Chowla sets in finite groups}

Throughout this section and the next, every finite group \(G\) is assumed to
satisfy \(|G|>1\).  This convention ensures that the family of Chowla sets is
nonempty and that the maxima below are well defined.

We write \(\ord(x)\) for the order of an element \(x\) in a finite group.  The exponent of a finite group \(G\), denoted \(\Exp(G)\), is the least positive integer \(e\) such that \(x^e=1\) for every \(x\in G\).  Equivalently, it is the least common multiple of the orders of the elements of \(G\).

\begin{definition}
Let \(G\) be a finite group.  A nonempty subset \(S\subseteq G\) is called a \emph{Chowla set} if
\[
 \ord(x)>|S| \qquad \text{for every }x\in S.
\]
Define
\[
 \Ccal(G)=\max\{|S|:S\subseteq G \text{ is a Chowla set}\}.
\]
For \(m\ge 0\), set
\[
 \Omega_G(m)=\{x\in G:\ord(x)>m\},
 \qquad
 N_G(m)=|\Omega_G(m)|.
\]
\end{definition}

The first observation reduces the invariant to the order distribution of the group.  In this paper, the order distribution of a finite group means the sequence of integers \(a_G(d)=|\{x\in G:\ord(x)=d\}|\), indexed by the positive divisors of \(\Exp(G)\).

\begin{proposition}\label{prop:order-distribution}\label{thm:order-distribution}
For every finite group \(G\),
\[
 \Ccal(G)=\max\{m\in\N:N_G(m)\ge m\}.
\]
Moreover, for each \(m\ge 1\), the number of Chowla sets of \(G\) of cardinality \(m\) is
\[
 \binom{N_G(m)}{m}.
\]
\end{proposition}

\begin{proof}
A subset \(S\subseteq G\) with \(|S|=m\) is Chowla if and only if every element of \(S\) lies in \(\Omega_G(m)\).  Hence such a subset exists if and only if \(N_G(m)\ge m\).  The counting statement follows from the same observation, since the Chowla sets of cardinality \(m\) are precisely the \(m\)-element subsets of \(\Omega_G(m)\).
\end{proof}

\begin{corollary}\label{cor:max-order-criterion}\label{prop:max-order-criterion}
Let \(G\) be a finite group and let \(e=\Exp(G)\).  If \(G\) contains at least \(e-1\) elements of order \(e\), then
\[
 \Ccal(G)=e-1.
\]
\end{corollary}

\begin{proof}
No element of \(G\) has order greater than \(e\), so no Chowla set can have cardinality at least \(e\).  Hence \(\Ccal(G)\le e-1\).  If \(G\) has at least \(e-1\) elements of order \(e\), then any choice of \(e-1\) such elements is a Chowla set.  Thus \(\Ccal(G)\ge e-1\).
\end{proof}

\subsection{Cyclic groups}

We now specialize to \(\Z/n\Z\).  The order distribution is explicit, so Proposition~\ref{prop:order-distribution} becomes a divisor formula.

\begin{theorem}\label{thm:cyclic-exact}
For every integer \(n\ge 2\),
\[
 \Ccal(\Z/n\Z)=\max\left\{m\in\N:\sum_{\substack{d\mid n\\ d>m}}\varphi(d)\ge m\right\}.
\]
In particular,
\[
 \varphi(n)\le \Ccal(\Z/n\Z)\le n-1.
\]
\end{theorem}

\begin{proof}
In the cyclic group \(\Z/n\Z\), the number of elements of order \(d\) is \(\varphi(d)\) for each divisor \(d\mid n\).  Hence
\[
 N_{\Z/n\Z}(m)=\sum_{\substack{d\mid n\\ d>m}}\varphi(d).
\]
The formula follows from Proposition~\ref{prop:order-distribution}.  The lower bound follows by taking all elements of order \(n\), and the upper bound follows because no element has order greater than \(n\).
\end{proof}

\begin{proposition}[Threshold form]\label{prop:cyclic-greedy}
Let \(n\ge 2\) and let \(m\) be an integer satisfying
\[
 \sum_{\substack{d\mid n\\ d>m}}\varphi(d)\ge m.
\]
Then there is a Chowla set of \(\Z/n\Z\) of cardinality \(m\) with the following form: for some divisor \(D\mid n\) with \(D>m\), the set contains all elements of order strictly larger than \(D\), exactly
\[
 m-\sum_{\substack{d\mid n\\ d>D}}\varphi(d)
\]
elements of order \(D\), and no element of order smaller than \(D\).  If \(D\) is taken to be the minimum order occurring in the set, this threshold description is unique up to the choice of the required elements of order \(D\).  Consequently, a maximum Chowla set may be chosen in this threshold form.
\end{proposition}

\begin{proof}
List all elements of \(\Z/n\Z\) having order greater than \(m\), arranged so that their orders are nonincreasing.  Since the displayed inequality holds, at least \(m\) such elements exist.  Choose the first \(m\) elements in this list and let \(D\) be the order of the last chosen element.  Then all elements of order larger than \(D\) have been chosen, exactly
\[
 m-\sum_{\substack{d\mid n\\ d>D}}\varphi(d)
\]
elements of order \(D\) have been chosen, and no element of order smaller than \(D\) has been chosen.  Thus, once the minimum occurring order \(D\) is fixed, only the choice among the elements of order \(D\) remains.  Since every chosen element has order greater than \(m\), the set is Chowla.  Taking \(m=\Ccal(\Z/n\Z)\) gives the final assertion.
\end{proof}

\begin{proposition}\label{prop:cyclic-phi-criterion}
Let
\[
 n=p_1^{\alpha_1}\cdots p_k^{\alpha_k},
 \qquad p_1<\cdots<p_k.
\]
Then
\[
 \Ccal(\Z/n\Z)=\varphi(n)
\]
if and only if
\[
 \varphi(n)\ge \frac{n}{p_1}-1.
\]
Equivalently,
\[
 \prod_{i=1}^k\left(1-\frac1{p_i}\right)+\frac1n\ge \frac1{p_1}.
\]
\end{proposition}

\begin{proof}
The largest proper divisor of \(n\) is \(n/p_1\).  Suppose first that \(\varphi(n)\ge n/p_1-1\).  If \(m>\varphi(n)\), then \(m\ge n/p_1\).  Thus an element of order greater than \(m\) must have order \(n\).  There are only \(\varphi(n)\) such elements, so no Chowla set can have cardinality greater than \(\varphi(n)\).  Since the \(\varphi(n)\) generators form a Chowla set, \(\Ccal(\Z/n\Z)=\varphi(n)\).

Conversely, suppose \(\varphi(n)<n/p_1-1\), and put \(m=\varphi(n)+1\).  Then \(m<n/p_1\).  Hence all elements of orders \(n\) and \(n/p_1\) have order greater than \(m\).  Since there are \(\varphi(n)\) elements of order \(n\) and at least one element of order \(n/p_1\), there is a Chowla set of size \(m\).  Therefore \(\Ccal(\Z/n\Z)>\varphi(n)\).

The product form follows from Euler's formula
\[
 \varphi(n)=n\prod_{i=1}^k\left(1-\frac1{p_i}\right).
\]
\end{proof}

\begin{corollary}\label{cor:prime-power-cyclic}
If \(n\) is a prime power, then
\[
 \Ccal(\Z/n\Z)=\varphi(n).
\]
\end{corollary}

\begin{proof}
Write \(n=p^a\).  Then \(\varphi(n)=n-n/p\ge n/p-1\), so Proposition~\ref{prop:cyclic-phi-criterion} applies.
\end{proof}

\begin{corollary}\label{cor:odd-two-primes}
If \(n\) is odd and has exactly two distinct prime divisors, then
\[
 \Ccal(\Z/n\Z)=\varphi(n).
\]
\end{corollary}

\begin{proof}
Write \(n=p^a q^b\), where \(3\le p<q\) are primes.  Then
\[
 \left(1-\frac1p\right)\left(1-\frac1q\right)
 \ge
 \left(1-\frac13\right)\left(1-\frac15\right)
 =\frac{8}{15}>\frac13\ge \frac1p.
\]
Thus the product criterion in Proposition~\ref{prop:cyclic-phi-criterion} holds.
\end{proof}

\begin{corollary}\label{cor:even-two-primes}
If \(n\) is even and has exactly two distinct prime divisors, then
\[
 \Ccal(\Z/n\Z)=\frac n2-1.
\]
\end{corollary}

\begin{proof}
Write \(n=2^a p^b\), where \(p\) is an odd prime and \(a,b\ge 1\).  Since the largest proper divisor of \(n\) is \(n/2\), no Chowla set can have size at least \(n/2\).  Hence
\[
 \Ccal(\Z/n\Z)\le \frac n2-1.
\]
All elements of orders \(n\) and \(n/2\) have order greater than \(n/2-1\).  Therefore it is enough to show that
\[
 \varphi(n)+\varphi(n/2)\ge \frac n2-1.
\]
If \(a=1\), then
\[
 \varphi(n)+\varphi(n/2)=2p^{b-1}(p-1)\ge p^b-1=\frac n2-1.
\]
If \(a\ge 2\), then
\[
 \varphi(n)+\varphi(n/2)=3\cdot 2^{a-2}p^{b-1}(p-1)
 \ge 2^{a-1}p^b-1=\frac n2-1,
\]
because \(3(p-1)\ge 2p\) for every odd prime \(p\).  Thus one can choose \(n/2-1\) elements among those of orders \(n\) and \(n/2\), and they form a Chowla set.
\end{proof}

\begin{remark}\label{rem:least-prime-many-factors}
Proposition~\ref{prop:cyclic-phi-criterion} gives many further explicit families.  In particular, the equality
\[
 \Ccal(\Z/n\Z)=\varphi(n)
\]
holds whenever
\[
 \prod_{i=1}^k\left(1-\frac1{p_i}\right)\ge \frac1{p_1},
\]
because this condition implies the product criterion
\[
 \prod_{i=1}^k\left(1-\frac1{p_i}\right)+\frac1n\ge \frac1{p_1}.
\]
Thus, for any fixed least prime divisor \(p_1\), one obtains an explicit finite-prime-product test.  Numerical thresholds obtained by taking consecutive primes beginning with \(p_1\) are computational consequences of the exact criterion rather than structurally distinguished constants.
\end{remark}

For a set \(\mathcal B\subseteq\N\), write
\[
 \underline d(\mathcal B)=
 \liminf_{x\to\infty}\frac1x|\{n\le x:n\in\mathcal B\}|
\]
for its lower asymptotic density.

\begin{theorem}\label{thm:large-ratio-positive-density}
For every real number \(M>0\), the set
\[
 \mathcal B_M=\{n\ge 2:\Ccal(\Z/n\Z)\ge M\varphi(n)\}
\]
has positive lower asymptotic density.  In particular, it is infinite.
\end{theorem}

\begin{proof}
Let \(M>0\).  Choose an integer \(k\ge 1\) such that
\[
 H_k:=\sum_{i=1}^k\frac1i\ge M.
\]
Put
\[
 B=\lcm(1,2,\ldots,k)^2.
\]
Choose a multiple \(N_0\) of \(B\) such that
\begin{equation}\label{eq:small-totient-ratio}
 M\frac{\varphi(N_0)}{N_0}+\frac1{N_0}<\frac1k.
\end{equation}
This is possible because multiplying \(B\) by sufficiently many new primes makes \(\varphi(N_0)/N_0\) arbitrarily small and also makes \(N_0\) arbitrarily large.

We claim that every multiple \(n\) of \(N_0\) belongs to \(\mathcal B_M\).  Since \(B\mid n\), each integer \(1\le i\le k\) divides \(n\), and for every prime \(p\mid i\) the exponent of \(p\) in \(n\) is strictly larger than its exponent in \(i\).  Therefore
\begin{equation}\label{eq:phi-divisor-ratio}
 \varphi\left(\frac ni\right)=\frac{\varphi(n)}{i}
 \qquad(1\le i\le k).
\end{equation}
Moreover, since \(n\) is a multiple of \(N_0\), we have \(\varphi(n)/n\le \varphi(N_0)/N_0\) and \(1/n\le 1/N_0\).  Hence
\[
 M\frac{\varphi(n)}{n}+\frac1n<\frac1k.
\]
Equivalently,
\[
 M\varphi(n)+1<\frac nk.
\]
Let
\[
 r=\lceil M\varphi(n)\rceil.
\]
Then \(r<n/k\).  For each \(1\le i\le k\), all elements of order \(n/i\) in \(\Z/n\Z\) have order greater than \(r\).  By \eqref{eq:phi-divisor-ratio}, the number of elements whose orders are among
\[
 n,\frac n2,\ldots,\frac nk
\]
is
\[
 \sum_{i=1}^k\varphi\left(\frac ni\right)=\varphi(n)H_k\ge M\varphi(n).
\]
This number is an integer, so it is at least \(r\).  Thus there are at least \(r\) elements of \(\Z/n\Z\) of order greater than \(r\).  By Proposition~\ref{prop:order-distribution},
\[
 \Ccal(\Z/n\Z)\ge r\ge M\varphi(n).
\]
Thus every multiple of \(N_0\) belongs to \(\mathcal B_M\).  The multiples of \(N_0\) have natural density \(1/N_0\), so \(\mathcal B_M\) has positive lower asymptotic density.
\end{proof}

\begin{corollary}\label{cor:ratio-asymptotics}
The limit
\[
 \lim_{n\to\infty}\frac{\Ccal(\Z/n\Z)}{\varphi(n)}
\]
does not exist.  Moreover,
\[
 \liminf_{n\to\infty}\frac{\Ccal(\Z/n\Z)}{\varphi(n)}=1,
 \qquad
 \limsup_{n\to\infty}\frac{\Ccal(\Z/n\Z)}{\varphi(n)}=\infty.
\]
\end{corollary}

\begin{proof}
Along prime powers \(n=p^a\), Corollary~\ref{cor:prime-power-cyclic} gives
\[
 \frac{\Ccal(\Z/n\Z)}{\varphi(n)}=1.
\]
Since \(\Ccal(\Z/n\Z)\ge \varphi(n)\) for every \(n\ge 2\), the liminf is exactly \(1\).  On the other hand, Theorem~\ref{thm:large-ratio-positive-density} shows that for every \(M>0\) there are infinitely many \(n\) such that
\[
 \frac{\Ccal(\Z/n\Z)}{\varphi(n)}\ge M.
\]
Therefore the limsup is infinite.
\end{proof}

\begin{proposition}\label{prop:order-normalization}
We have
\[
 \liminf_{n\to\infty}\frac{\Ccal(\Z/n\Z)}{n}=0,
 \qquad
 \limsup_{n\to\infty}\frac{\Ccal(\Z/n\Z)}{n}=1.
\]
In particular, the limit \(\lim_{n\to\infty}\Ccal(\Z/n\Z)/n\) does not exist.
\end{proposition}

\begin{proof}
The upper bound \(\Ccal(\Z/n\Z)\le n-1\) holds for every \(n\ge 2\).  If \(n=p\) is prime, then every nonzero element of \(\Z/p\Z\) has order \(p\), and hence \(\Ccal(\Z/p\Z)=p-1\).  Therefore
\[
 \frac{\Ccal(\Z/p\Z)}p=1-\frac1p\longrightarrow 1
\]
along the primes, proving that the limsup is \(1\).

It remains to prove that the liminf is \(0\).  Let
\[
 n_x=\prod_{p\le x}p
\]
be the primorial up to \(x\).  We show that \(\Ccal(\Z/n_x\Z)/n_x\to 0\).  Fix \(\varepsilon>0\).  If \(d\mid n_x\) and \(d>\varepsilon n_x\), then \(d=n_x/e\) for some divisor \(e\mid n_x\) with \(e<1/\varepsilon\).  Hence
\[
 \sum_{\substack{d\mid n_x\\ d>\varepsilon n_x}}\varphi(d)
 \le
 \sum_{\substack{e\mid n_x\\ e<1/\varepsilon}}\varphi(n_x/e).
\]
There are only finitely many possible integers \(e<1/\varepsilon\).  For each fixed such \(e\), and for all sufficiently large \(x\),
\[
 \frac{\varphi(n_x/e)}{n_x}
 \le
 \frac{e\varphi(n_x/e)}{n_x}
 =
 \frac{\varphi(n_x/e)}{n_x/e}
 =
 \frac{\varphi(n_x)}{n_x}
 \prod_{p\mid e}\left(1-\frac1p\right)^{-1}.
\]
The last expression tends to \(0\), since
\[
 \frac{\varphi(n_x)}{n_x}=\prod_{p\le x}\left(1-\frac1p\right)\longrightarrow 0,
\]
by the divergence of the reciprocal prime series; see, for example, \cite[Theorem~19]{HardyWright2008}.  Therefore
\[
 \frac1{n_x}\sum_{\substack{d\mid n_x\\ d>\varepsilon n_x}}\varphi(d)\longrightarrow 0.
\]
For all sufficiently large \(x\), this sum is less than \(\varepsilon n_x\).  By Theorem~\ref{thm:cyclic-exact}, no Chowla set of \(\Z/n_x\Z\) can have cardinality at least \(\varepsilon n_x\).  Thus \(\Ccal(\Z/n_x\Z)/n_x<\varepsilon\) for all sufficiently large \(x\).  Since \(\varepsilon>0\) was arbitrary, the liminf is \(0\).
\end{proof}

We next record a qualitative density consequence for the equality \(\Ccal(\Z/n\Z)=\varphi(n)\).  The proof is elementary and uses only a finite-prime restriction together with Markov's inequality.

\begin{proposition}\label{prop:positive-lower-density}
The set
\[
 \mathcal A=\{n\ge 2:\Ccal(\Z/n\Z)=\varphi(n)\}
\]
has positive lower asymptotic density.
\end{proposition}

\begin{proof}
It is enough to produce a positive-density subfamily of \(\mathcal A\).  We consider odd multiples of \(3\).  By Proposition~\ref{prop:cyclic-phi-criterion}, an odd integer \(n\) divisible by \(3\) belongs to \(\mathcal A\) provided
\[
 \prod_{\substack{q\mid n\\ q>3}}\left(1-\frac1q\right)+\frac{3}{2n}\ge \frac12.
\]
Thus it suffices to find a positive-density set of odd multiples of \(3\) for which the product over prime divisors \(q>3\) is larger than \(1/2\).

Choose \(Y>3\) so large that
\[
 \sum_{q>Y}\frac1q\log\left(\frac{q}{q-1}\right)<\frac{\log 2}{2}.
\]
Let \(\mathcal R_Y\) be the set of integers \(n\) such that \(3\mid n\), \(2\nmid n\), and no prime \(q\) with \(3<q\le Y\) divides \(n\).  This set has positive natural density, namely
\[
 \frac16\prod_{3<q\le Y}\left(1-\frac1q\right).
\]
For \(n\in \mathcal R_Y\), set
\[
 f_Y(n)=\sum_{\substack{q\mid n\\ q>Y}}\log\left(\frac{q}{q-1}\right).
\]
We claim that the limsup of the mean value of \(f_Y\) over \(\mathcal R_Y\) is at most
\[
 \sum_{q>Y}\frac1q\log\left(\frac{q}{q-1}\right).
\]
Indeed, after imposing the finitely many congruence restrictions defining \(\mathcal R_Y\), the divisibility condition \(q\mid n\) for a prime \(q>Y\) has relative density \(1/q\).  Truncating the sum to \(Y<q\le Z\), averaging over \(\mathcal R_Y\), and then letting \(Z\to\infty\) gives the claim, since the displayed prime series converges.

Therefore the limsup mean of \(f_Y\) over \(\mathcal R_Y\) is less than \(\log 2/2\).  By Markov's inequality, a subset of \(\mathcal R_Y\) of positive lower density satisfies
\[
 f_Y(n)<\log 2.
\]
For such \(n\), the primes \(3<q\le Y\) do not divide \(n\), and hence
\[
 \sum_{\substack{q\mid n\\ q>3}}\log\left(\frac{q}{q-1}\right)
 =
 f_Y(n)
 <\log 2.
\]
Equivalently,
\[
 \prod_{\substack{q\mid n\\ q>3}}\left(1-\frac1q\right)>\frac12.
\]
For these integers \(n\), Proposition~\ref{prop:cyclic-phi-criterion} gives \(\Ccal(\Z/n\Z)=\varphi(n)\).  Hence \(\mathcal A\) has positive lower asymptotic density.
\end{proof}

\begin{remark}
The proof of Proposition~\ref{prop:positive-lower-density} avoids the full Erd\H{o}s--Wintner theorem.  Nevertheless, that theorem gives a useful conceptual explanation for the density questions below.  Recall that an arithmetic function \(f\) is additive if \(f(ab)=f(a)+f(b)\) whenever \(\gcd(a,b)=1\).  After fixing the least prime divisor \(p\), the condition in Proposition~\ref{prop:cyclic-phi-criterion} is asymptotically governed by the additive arithmetic function
\[
 \sum_{\substack{q\mid n\\ q>p}}\log\left(\frac{q}{q-1}\right).
\]
The corresponding prime series is convergent, so the Erd\H{o}s--Wintner theorem implies that this function has a limiting distribution; see \cite{ErdosWintner1939,Tenenbaum1995}.  This suggests that the equality condition \(\Ccal(\Z/n\Z)=\varphi(n)\) should have a natural density.
\end{remark}

\begin{question}\label{ques:large-ratio-density}
For \(M>0\), determine or estimate the lower asymptotic density of
\[
 \{n\ge 2:\Ccal(\Z/n\Z)\ge M\varphi(n)\}.
\]
More generally, obtain effective bounds for \(\Ccal(\Z/n\Z)/\varphi(n)\) in terms of the small prime divisors of \(n\) and their exponents.
\end{question}

We use the term natural density for a limit of the form
\[
 \lim_{x\to\infty}\frac1x|\{n\le x:n\in \mathcal S\}|,
\]
when this limit exists.

\begin{question}\label{ques:density-equality}
Does the natural density
\[
 \delta=\lim_{x\to\infty}\frac1x|\{n\le x:\Ccal(\Z/n\Z)=\varphi(n)\}|
\]
exist?  If it exists, determine or estimate \(\delta\).  Equivalently, determine the density of integers \(n\) for which
\[
 \frac{\varphi(n)}n+\frac1n\ge \frac1{p(n)},
\]
where \(p(n)\) is the least prime divisor of \(n\).
\end{question}

\section{Finite abelian groups}

We now pass from cyclic groups to finite abelian groups.  We write \(C_m\) for a cyclic group of order \(m\).  We use the standard classification of finite abelian groups, according to which every finite abelian group is isomorphic to
\[
 G\cong C_{n_1}\oplus\cdots\oplus C_{n_r},
 \qquad n_1\mid n_2\mid\cdots\mid n_r,
\]
for suitable positive integers \(n_i\); see, for example, \cite[Sec.~5.2]{DummitFoote2004}.  The only additional ingredient is the standard count of elements whose order divides a given integer.

\begin{lemma}\label{lem:dividing-order-count}
Let
\[
 G\cong C_{n_1}\oplus\cdots\oplus C_{n_r},
 \qquad n_1\mid n_2\mid\cdots\mid n_r.
\]
For \(d\ge 1\), set
\[
 T_G(d)=|\{x\in G:\ord(x)\mid d\}|.
\]
Then
\[
 T_G(d)=\prod_{i=1}^r\gcd(d,n_i).
\]
\end{lemma}

\begin{proof}
Write \(x=(x_1,\ldots,x_r)\) with \(x_i\in C_{n_i}\).  The condition \(\ord(x)\mid d\) is equivalent to \(x_i^d=1\) for all \(i\).  In a cyclic group of order \(n_i\), the equation \(y^d=1\) has exactly \(\gcd(d,n_i)\) solutions.  Multiplying over the direct factors gives the formula.
\end{proof}

\begin{theorem}\label{thm:abelian-formula}
Let
\[
 G\cong C_{n_1}\oplus\cdots\oplus C_{n_r},
 \qquad n_1\mid n_2\mid\cdots\mid n_r,
\]
and let \(e=n_r=\Exp(G)\).  For each divisor \(d\mid e\), let
\[
 a_G(d)=|\{x\in G:\ord(x)=d\}|.
\]
Then
\[
 a_G(d)=\sum_{c\mid d}\mu\!\left(\frac dc\right)\prod_{i=1}^r\gcd(c,n_i).
\]
Consequently,
\[
 \Ccal(G)=\max\left\{m\in\N:\sum_{\substack{d\mid e\\ d>m}}a_G(d)\ge m\right\}.
\]
\end{theorem}

\begin{proof}
For every \(d\mid e\),
\[
 T_G(d)=\sum_{c\mid d}a_G(c).
\]
Here \(\mu\) denotes the classical number-theoretic Möbius function.  Möbius inversion on the divisor lattice, together with Lemma~\ref{lem:dividing-order-count}, gives the displayed formula for \(a_G(d)\); see, for example, \cite[Ch.~3]{Stanley2011}.  Since all element orders in \(G\) divide \(e\), we have
\[
 N_G(m)=\sum_{\substack{d\mid e\\ d>m}}a_G(d).
\]
The result follows from Proposition~\ref{prop:order-distribution}.
\end{proof}

\begin{theorem}\label{thm:pgroup}
Let \(G\) be a finite abelian \(p\)-group of exponent \(p^\ell\).  Then
\[
 \Ccal(G)=
 \begin{cases}
 \varphi(p^\ell)=p^\ell-p^{\ell-1}, & \text{if }G\text{ is cyclic},\\[3pt]
 p^\ell-1, & \text{if }G\text{ is noncyclic}.
 \end{cases}
\]
\end{theorem}

\begin{proof}
If \(G\) is cyclic, this is Corollary~\ref{cor:prime-power-cyclic}.

Assume now that \(G\) is noncyclic.  Since \(\Exp(G)=p^\ell\), no Chowla set can have cardinality at least \(p^\ell\).  Hence
\[
 \Ccal(G)\le p^\ell-1.
\]
It remains to prove the reverse inequality.

By the classification of finite abelian \(p\)-groups, \(G\) has a direct factor isomorphic to \(C_{p^\ell}\).  Since \(G\) is not cyclic, it has another nontrivial cyclic direct factor, say \(C_{p^k}\) with \(k\ge 1\).  The latter contains a subgroup isomorphic to \(C_p\).  Therefore \(G\) contains a subgroup
\[
 H\cong C_{p^\ell}\oplus C_p .
\]
Every element of \(H\) whose first coordinate has order \(p^\ell\) has order \(p^\ell\).  There are
\[
 p\varphi(p^\ell)=p(p^\ell-p^{\ell-1})=p^\ell(p-1)
\]
such elements.  Since \(p^\ell(p-1)\ge p^\ell-1\), the group \(G\) contains at least \(p^\ell-1\) elements of order \(p^\ell\).  Choosing any \(p^\ell-1\) of them gives a Chowla set.  Thus
\[
 \Ccal(G)\ge p^\ell-1,
\]
and equality follows.
\end{proof}

\begin{corollary}\label{cor:elementary-abelian}
For every prime \(p\) and every integer \(r\ge 1\),
\[
 \Ccal((C_p)^r)=p-1.
\]
\end{corollary}

\begin{proof}
This is the case \(\ell=1\) of Theorem~\ref{thm:pgroup}.
\end{proof}

\section{Chowla subspaces in finite extensions}\label{sec:chowla_subspaces_in_finite_extensions}

We now pass from element orders to degrees of generated subextensions. The
resulting condition depends on the dimension of the subspace and therefore
retains more information than the requirement that all nonzero elements
generate the entire extension.

\begin{definition}
Let \(K\subseteq L\) be a finite extension with \([L:K]=n>1\). A nonzero
\(K\)-subspace \(A\subseteq L\) is called a \emph{Chowla subspace} if
\[
 [K(a):K]>\dim_K A
 \qquad\text{for every }0\ne a\in A.
\]
Define
\[
 \Ccal(L/K)=
 \max\{\dim_K A:A\subseteq L\text{ is a Chowla }K\text{-subspace}\}.
\]
Finally, set
\[
 d_{\max}(L/K)=
 \max\{[E:K]:K\subseteq E\subsetneq L\}.
\]
\end{definition}

\begin{proposition}[Degree-threshold characterization]\label{prop:degree-threshold}
Let \(A\subseteq L\) be a nonzero \(K\)-subspace and put
\(m=\dim_K A\). Then \(A\) is a Chowla subspace if and only if
\[
 A\cap E=\{0\}
\]
for every intermediate field \(E\) satisfying \([E:K]\le m\).
\end{proposition}

\begin{proof}
Suppose first that \(A\) is Chowla. If \(0\ne a\in A\cap E\), then
\([K(a):K]\le [E:K]\le m\), contradicting the defining inequality.
Conversely, if \(A\) is not Chowla, then there is a nonzero \(a\in A\) with
\([K(a):K]\le m\). The intermediate field \(E=K(a)\) then satisfies
\(0\ne a\in A\cap E\), so the intersection condition fails.
\end{proof}

\begin{remark}\label{rem:chowla-not-full-generation}
The Chowla condition does not generally require every nonzero element to
generate \(L\). For example, suppose \(L/K\) is finite separable and has a
proper intermediate field \(E\) of degree \(r>1\). Choose \(a\in E\) with
\(E=K(a)\). Then the one-dimensional space \(Ka\) is Chowla because
\([K(a):K]=r>1\), although \(K(a)=E\ne L\).
\end{remark}

\begin{theorem}\label{thm:linear-upper-bound}
Let \(K\subseteq L\) be a finite extension of degree \(n>1\). Then
\[
 \Ccal(L/K)\le n-d_{\max}(L/K).
\]
\end{theorem}

\begin{proof}
Let \(A\subseteq L\) be a Chowla subspace and put \(m=\dim_K A\). Choose a
proper intermediate field \(E\) with \([E:K]=d_{\max}(L/K)\), and abbreviate
this degree by \(d\). Since \([L:E]\ge2\), we have \(d\le n/2\), and hence
\(n-d\ge d\).

If \(m>n-d\), then
\[
 \dim_K(A\cap E)\ge m+d-n>0.
\]
Choose \(0\ne a\in A\cap E\). Then
\[
 [K(a):K]\le d\le n-d<m,
\]
contradicting the Chowla condition. Therefore \(m\le n-d\).
\end{proof}

The lower bound requires constructing a subspace that avoids all relevant
intermediate fields. Over an infinite base field, this follows from two
standard finite-dimensional facts.

\begin{lemma}\label{lem:finitely-many-intermediate-fields}
If \(K\subseteq L\) is finite and separable, then there are only finitely many
intermediate fields between \(K\) and \(L\).
\end{lemma}

\begin{proof}
Let \(M\) be a normal closure of \(L/K\). Then \(M/K\) is finite Galois. By
the finite Galois correspondence \cite[Sec.~14.2]{DummitFoote2004}, the map
\[
 E\longmapsto \Gal(M/E)
\]
embeds the set of intermediate fields \(K\subseteq E\subseteq L\) into the
finite set of subgroups of \(\Gal(M/K)\).
\end{proof}

\begin{lemma}\label{lem:finite-union-subspaces}
Let \(K\) be an infinite field and let \(V\) be a finite-dimensional
\(K\)-vector space. Then \(V\) is not a finite union of proper
\(K\)-subspaces.
\end{lemma}

\begin{proof}
We argue by induction on \(\dim_K V\). The one-dimensional case is immediate.
Suppose \(\dim_K V\ge2\) and
\(V=U_1\cup\cdots\cup U_t\), with each \(U_i\) proper. Choose a hyperplane
\(H\) different from every \(U_i\) that is itself a hyperplane; this is
possible because \(K\) is infinite. Then every \(H\cap U_i\) is a proper
subspace of \(H\), while
\[
 H=\bigcup_{i=1}^t(H\cap U_i),
\]
contradicting the induction hypothesis.
\end{proof}

\begin{theorem}[Exact formula over infinite fields]\label{thm:infinite-sharp}
Let \(K\subseteq L\) be a finite separable extension of degree \(n>1\), and
assume that \(K\) is infinite. Then
\[
 \boxed{\Ccal(L/K)=n-d_{\max}(L/K).}
\]
\end{theorem}

\begin{proof}
The upper bound is Theorem~\ref{thm:linear-upper-bound}. Put
\[
 d=d_{\max}(L/K),\qquad m=n-d.
\]
Let \(E_1,\ldots,E_t\) be the maximal proper intermediate fields. By
Lemma~\ref{lem:finitely-many-intermediate-fields}, there are finitely many,
and every proper intermediate field is contained in one of them.

We construct subspaces
\[
 A_0\subset A_1\subset\cdots\subset A_m
\]
with \(\dim_K A_j=j\) and \(A_j\cap E_i=\{0\}\) for every \(i\). Begin with
\(A_0=\{0\}\). Suppose \(A_j\) has been constructed for some \(j<m\). For
each \(i\),
\[
 \dim_K(E_i+A_j)\le d+j\le d+m-1=n-1,
\]
so \(E_i+A_j\) is proper. Lemma~\ref{lem:finite-union-subspaces} provides
\[
 v\in L\setminus\bigcup_{i=1}^t(E_i+A_j).
\]
Set \(A_{j+1}=A_j\oplus Kv\). If \(a+cv\in A_{j+1}\cap E_i\), where
\(a\in A_j\), then \(c\ne0\) would imply \(v\in E_i+A_j\). Hence \(c=0\),
and then \(a=0\). Thus \(A_{j+1}\cap E_i=\{0\}\) for every \(i\).

Let \(A=A_m\). Every nonzero \(a\in A\) lies in no proper intermediate
field, so \(K(a)=L\). Therefore
\[
 [K(a):K]=n>m=\dim_K A,
\]
and \(A\) is Chowla. Hence \(\Ccal(L/K)\ge m\), completing the proof.
\end{proof}

\begin{remark}\label{rem:extremal-full-degree}
Let \(d=d_{\max}(L/K)\). Since \(d\le n/2\), the extremal dimension
\(n-d\) is at least \(d\). Consequently, any Chowla subspace of dimension
\(n-d\) automatically satisfies \(K(a)=L\) for every \(0\ne a\in A\): a
proper generated subextension would have degree at most \(d\le n-d\). This
is a consequence of extremality and is not part of the definition of a Chowla
subspace.
\end{remark}

\begin{corollary}\label{cor:prime-degree-infinite}
If \(K\subseteq L\) is a finite separable extension of prime degree \(n\) and
\(K\) is infinite, then
\[
 \Ccal(L/K)=n-1.
\]
\end{corollary}

\begin{proof}
There is no proper intermediate field strictly larger than \(K\), so
\(d_{\max}(L/K)=1\). Apply Theorem~\ref{thm:infinite-sharp}.
\end{proof}

\section{Finite fields}

Throughout this section, \(q\) and \(Q\) denote prime powers. We regard
\(\F_{q^n}\) as an \(n\)-dimensional vector space over \(\F_q\). The
intermediate fields of \(\F_{q^n}/\F_q\) are precisely the fields
\(\F_{q^e}\) with \(e\mid n\), and
\[
 \F_{q^a}\cap\F_{q^b}=\F_{q^{\gcd(a,b)}}.
\]
We also use the normal basis theorem; see \cite[Ch.~2]{LidlNiederreiter1997}.

\subsection{The exact extremal formula}

The upper bound from Theorem~\ref{thm:linear-upper-bound} is sharp over every
finite field. The proof below is explicit and works for arbitrary extension
degree.

\begin{theorem}[Finite-field formula]\label{thm:finite-field-exact}
Let \(n>1\), let \(p\) be the least prime divisor of \(n\), and let \(q\) be a
prime power. Then
\[
 \boxed{\Ccal(\F_{q^n}/\F_q)=n-\frac np.}
\]
\end{theorem}

\begin{proof}
The largest proper divisor of \(n\) is \(d=n/p\), so the largest proper
intermediate field has degree \(d\). Theorem~\ref{thm:linear-upper-bound}
gives
\[
 \Ccal(\F_{q^n}/\F_q)\le n-d.
\]

Choose a normal element \(\alpha\in\F_{q^n}\). Then
\[
 \alpha,\alpha^q,\ldots,\alpha^{q^{n-1}}
\]
is an \(\F_q\)-basis of \(\F_{q^n}\). Write
\(\alpha_i=\alpha^{q^i}\), with indices taken modulo \(n\), and set
\[
 A=\operatorname{span}_{\F_q}\{\alpha_0,\ldots,\alpha_{n-d-1}\}.
\]
Thus \(\dim_{\F_q}A=n-d\).

Suppose that \(0\ne x\in A\) belongs to a proper intermediate field
\(\F_{q^e}\), where \(e\mid n\) and \(e<n\). Since \(d\) is the largest
proper divisor of \(n\), we have \(e\le d\). Write
\[
 x=\sum_{i=0}^{n-1}c_i\alpha_i,
 \qquad c_i\in\F_q.
\]
By the definition of \(A\),
\[
 c_{n-d}=c_{n-d+1}=\cdots=c_{n-1}=0.
\]
On the other hand, \(x\in\F_{q^e}\) implies \(x^{q^e}=x\). Since
\(\alpha_i^{q^e}=\alpha_{i+e}\), uniqueness of normal-basis coordinates gives
\[
 c_i=c_{i-e}\qquad(i\bmod n).
\]
Hence the coordinate sequence is periodic with period \(e\). The displayed
block contains at least \(e\) consecutive zero coordinates, one from every
residue class modulo \(e\). Periodicity therefore forces all \(c_i\) to be
zero, contradicting \(x\ne0\).

Thus \(A\) meets every proper intermediate field trivially. Every nonzero
\(x\in A\) therefore satisfies \(\F_q(x)=\F_{q^n}\), and
\[
 [\F_q(x):\F_q]=n>n-d=\dim_{\F_q}A.
\]
So \(A\) is Chowla and \(\Ccal(\F_{q^n}/\F_q)\ge n-d\).
\end{proof}

\begin{corollary}[Exact formula for finite separable extensions]\label{cor:separable-exact}
Let \(K\subseteq L\) be a finite separable extension of degree \(n>1\). Then
\[
 \boxed{\Ccal(L/K)=n-d_{\max}(L/K).}
\]
\end{corollary}

\begin{proof}
If \(K\) is infinite, this is Theorem~\ref{thm:infinite-sharp}. If \(K\) is
finite, then \(K=\F_q\) and \(L=\F_{q^n}\) for some prime power \(q\), and
the result is Theorem~\ref{thm:finite-field-exact}.
\end{proof}

\subsection{Counting in the chain case}

We write
\[
 \gbinom{a}{b}{q}
\]
for the Gaussian binomial coefficient, the number of \(b\)-dimensional
subspaces of an \(a\)-dimensional vector space over \(\F_q\). It is understood
to be zero when \(b<0\) or \(b>a\).

\begin{lemma}\label{lem:disjoint-count}
Let \(V\) be an \(n\)-dimensional vector space over \(\F_q\), and let
\(W\subseteq V\) be a fixed \(d\)-dimensional subspace.  Then the number of
\(m\)-dimensional subspaces \(A\subseteq V\) satisfying
\[
 A\cap W=\{0\}
\]
is
\[
 q^{dm}\gbinom{n-d}{m}{q}.
\]
\end{lemma}

\begin{proof}
Let \(\pi:V\to V/W\) be the quotient map.  If \(A\cap W=\{0\}\), then
\(\pi|_A\) is injective, and hence \(\pi(A)\) is an \(m\)-dimensional subspace
of \(V/W\).

Conversely, fix an \(m\)-dimensional subspace \(B\subseteq V/W\).  The
subspaces \(A\subseteq V\) satisfying \(A\cap W=\{0\}\) and \(\pi(A)=B\) are
precisely the graphs of \(\F_q\)-linear maps \(B\to W\), after identifying
\(\pi^{-1}(B)\) with \(B\oplus W\).  Since
\[
 \dim_{\F_q}B=m
 \qquad\text{and}\qquad
 \dim_{\F_q}W=d,
\]
there are \(q^{dm}\) such linear maps.  There are
\[
 \gbinom{n-d}{m}{q}
\]
choices for \(B\subseteq V/W\).  Multiplying these two numbers gives the
formula.
\end{proof}

\begin{proposition}\label{prop:unique-maximal-forbidden}
Let \(1\le m<n\).  Suppose that among the divisors \(d\mid n\) satisfying
\(d\le m\), there is a unique maximal element \(d_0\) under divisibility.
Then the number of \(m\)-dimensional Chowla \(\F_q\)-subspaces of
\(\F_{q^n}/\F_q\) is
\[
 q^{d_0m}\gbinom{n-d_0}{m}{q}.
\]
\end{proposition}

\begin{proof}
An \(m\)-dimensional subspace \(A\subseteq\F_{q^n}\) is Chowla if and only if
it contains no nonzero element lying in a proper subfield \(\F_{q^d}\) with
\(d\mid n\) and \(d\le m\).  Since \(d_0\) is the unique maximal such divisor,
all these forbidden subfields are contained in \(\F_{q^{d_0}}\).  Therefore
\(A\) is Chowla if and only if
\[
 A\cap\F_{q^{d_0}}=\{0\}.
\]
The field \(\F_{q^{d_0}}\) has dimension \(d_0\) over \(\F_q\), so the result
follows from Lemma~\ref{lem:disjoint-count}.
\end{proof}

\begin{theorem}\label{thm:prime-power-count}
Let \(n=p^r\), where \(p\) is prime and \(r\ge 1\).  Let \(1\le m<n\), and
choose \(j\) such that
\[
 p^j\le m<p^{j+1},
 \qquad 0\le j\le r-1.
\]
Then the number of \(m\)-dimensional Chowla \(\F_q\)-subspaces of
\(\F_{q^n}/\F_q\) is
\[
 q^{p^jm}\gbinom{p^r-p^j}{m}{q}.
\]
\end{theorem}

\begin{proof}
The divisors of \(p^r\) not exceeding \(m\) are precisely
\[
 1,p,\ldots,p^j,
\]
and they form a chain under divisibility.  Hence the unique maximal forbidden
subfield is \(\F_{q^{p^j}}\).  Applying
Proposition~\ref{prop:unique-maximal-forbidden} with \(d_0=p^j\) gives
\[
 q^{p^jm}\gbinom{p^r-p^j}{m}{q}.
\]
\end{proof}

\begin{corollary}\label{cor:finite-field-prime-power}
If \(n=p^r\), where \(p\) is prime, then
\[
 \Ccal(\F_{q^n}/\F_q)=n-\frac np.
\]
\end{corollary}

\begin{proof}
This is the special case \(n=p^r\) of Theorem~\ref{thm:finite-field-exact}.
\end{proof}

\subsection{The first non-chain case}

Let \(n=pq\), where \(p<q\) are primes. The proper intermediate fields
\(\F_{Q^p}\) and \(\F_{Q^q}\) are incomparable, so a single forbidden-subspace
count no longer suffices.

\begin{corollary}\label{cor:pq-finite-field-extremal}
Let \(p<q\) be primes and let \(Q\) be a prime power. Then
\[
 \Ccal(\F_{Q^{pq}}/\F_Q)=pq-q=q(p-1).
\]
\end{corollary}

\begin{proof}
The least prime divisor of \(pq\) is \(p\). Apply
Theorem~\ref{thm:finite-field-exact}.
\end{proof}

\begin{theorem}\label{thm:pq-count}
Let \(p<q\) be primes, let \(Q\) be a prime power, and let
\(L=\F_{Q^{pq}}\).  For \(q\le m<pq\), the number of \(m\)-dimensional
Chowla \(\F_Q\)-subspaces of \(L\) is
\[
 \sum_{a=0}^{q}\sum_{b=0}^{p}(-1)^{a+b}
 Q^{\binom{a}{2}+\binom{b}{2}}
 \left(
 N_{a,b}^{(0)}\gbinom{pq-a-b}{m-a-b}{Q}
 +N_{a,b}^{(1)}\gbinom{pq-a-b+1}{m-a-b+1}{Q}
 \right),
\]
where
\[
 N_{a,b}^{(1)}
 =
 \gbinom{q-1}{a-1}{Q}
 \gbinom{p-1}{b-1}{Q}
\]
and
\[
 N_{a,b}^{(0)}
 =
 \gbinom{q}{a}{Q}
 \gbinom{p}{b}{Q}
 -
 \gbinom{q-1}{a-1}{Q}
 \gbinom{p-1}{b-1}{Q}.
\]
As usual, Gaussian binomial coefficients with negative lower index are
interpreted as zero.
\end{theorem}

\begin{proof}
Put
\[
 K=\F_Q,\qquad E=\F_{Q^q},\qquad F=\F_{Q^p}.
\]
The proper intermediate fields of \(L/K\) are exactly
\[
 K,\qquad E,\qquad F.
\]
Since \(q\le m<pq\), an \(m\)-dimensional subspace \(A\subseteq L\) is Chowla
if and only if
\[
 A\cap E=A\cap F=\{0\}.
\]
Indeed, \(K\subseteq E\cap F\), so avoiding \(E\) and \(F\) also avoids \(K\).

Here \(\mu_E\) and \(\mu_F\) denote the Möbius functions of the subspace
lattices of \(E\) and \(F\), respectively.  For a subspace \(A\subseteq L\),
the indicator of the condition \(A\cap E=\{0\}\) is
\[
 \sum_{U\subseteq A\cap E}\mu_E(0,U),
\]
and similarly the indicator of \(A\cap F=\{0\}\) is
\[
 \sum_{W\subseteq A\cap F}\mu_F(0,W).
\]
Therefore the desired number is
\[
 \sum_{U\subseteq E}\sum_{W\subseteq F}
 \mu_E(0,U)\mu_F(0,W)
 \#\{A\le L:\dim_K A=m,\ U+W\subseteq A\}.
\]

Let
\[
 \dim_K U=a,\qquad \dim_K W=b,\qquad \dim_K(U\cap W)=\varepsilon.
\]
Then
\[
 \dim_K(U+W)=a+b-\varepsilon.
\]
Thus the number of \(m\)-dimensional subspaces \(A\subseteq L\) containing
\(U+W\) is
\[
 \gbinom{pq-a-b+\varepsilon}{m-a-b+\varepsilon}{Q}.
\]
Also, the Möbius functions of finite subspace lattices satisfy
\[
 \mu_E(0,U)=(-1)^aQ^{\binom{a}{2}},
 \qquad
 \mu_F(0,W)=(-1)^bQ^{\binom{b}{2}}.
\]

It remains to count the pairs \((U,W)\) according to whether
\(\varepsilon=0\) or \(\varepsilon=1\).  Since
\[
 E\cap F=K
\]
and \(K\) is one-dimensional over itself, the intersection \(U\cap W\) is
either zero or the common line \(K\).  Hence \(\varepsilon=1\) occurs exactly
when both \(U\) and \(W\) contain \(K\).  Therefore the number of pairs
\((U,W)\) with
\[
 \dim_K U=a,\qquad \dim_K W=b,\qquad \dim_K(U\cap W)=1
\]
is
\[
 N_{a,b}^{(1)}
 =
 \gbinom{q-1}{a-1}{Q}
 \gbinom{p-1}{b-1}{Q}.
\]
The total number of pairs \((U,W)\) with \(\dim_K U=a\) and \(\dim_K W=b\) is
\[
 \gbinom{q}{a}{Q}\gbinom{p}{b}{Q}.
\]
Hence the number of such pairs with \(U\cap W=\{0\}\) is
\[
 N_{a,b}^{(0)}
 =
 \gbinom{q}{a}{Q}
 \gbinom{p}{b}{Q}
 -
 \gbinom{q-1}{a-1}{Q}
 \gbinom{p-1}{b-1}{Q}.
\]

Substituting the two cases \(\varepsilon=0\) and \(\varepsilon=1\) into the
double Möbius-inversion sum gives
\[
 \sum_{a=0}^{q}\sum_{b=0}^{p}(-1)^{a+b}
 Q^{\binom{a}{2}+\binom{b}{2}}
 \left(
 N_{a,b}^{(0)}\gbinom{pq-a-b}{m-a-b}{Q}
 +N_{a,b}^{(1)}\gbinom{pq-a-b+1}{m-a-b+1}{Q}
 \right).
\]
This is the claimed formula.
\end{proof}

\begin{remark}
For \(1\le m<p\), the only forbidden proper subfield is \(\F_Q\).  For
\(p\le m<q\), the unique maximal forbidden subfield is \(\F_{Q^p}\).  Hence
Proposition~\ref{prop:unique-maximal-forbidden} gives the corresponding
counting formulas.

The range \(q\le m<pq\) is different: one must avoid two incomparable proper
subfields, \(\F_{Q^q}\) and \(\F_{Q^p}\), whose intersection is \(\F_Q\).
Thus the one-forbidden-subspace quotient-and-graph count from
Lemma~\ref{lem:disjoint-count} does not directly apply, and the
inclusion--exclusion formula in Theorem~\ref{thm:pq-count} is the natural
replacement.
\end{remark}

\section{Further questions}

The cyclic formula in Theorem~\ref{thm:cyclic-exact} reduces the computation of \(\Ccal(\Z/n\Z)\) to the divisor structure of \(n\).  Proposition~\ref{prop:cyclic-phi-criterion} gives an exact numerical criterion for the equality \(\Ccal(\Z/n\Z)=\varphi(n)\), but a more conceptual classification of the integers satisfying this equality remains desirable.

Theorem~\ref{thm:large-ratio-positive-density} and Corollary~\ref{cor:ratio-asymptotics} show that the ratio \(\Ccal(\Z/n\Z)/\varphi(n)\) has liminf \(1\) and infinite limsup, and that every fixed large-value threshold is attained on a set of positive lower density.  It would be interesting to estimate these densities effectively as the threshold grows.

Proposition~\ref{prop:order-normalization} shows that
\[
 \liminf_{n\to\infty}\frac{\Ccal(\Z/n\Z)}n=0,
 \qquad
 \limsup_{n\to\infty}\frac{\Ccal(\Z/n\Z)}n=1.
\]
Thus the normalization by the group order has a very different behavior from the normalization by \(\varphi(n)\).  One may ask for more precise distributional information about this bounded ratio, for example along integers with prescribed small prime divisors.

Finally, Corollary~\ref{cor:separable-exact} determines the extremal dimension for every finite separable extension, but the enumeration problem remains open in general. Over finite fields it becomes increasingly combinatorial when the lattice of intermediate fields is not a chain. The case \(n=pq\) is handled by two-subfield inclusion--exclusion. For degrees with three or more incomparable maximal proper divisors, one expects higher-order formulas governed by the intersection pattern of the corresponding subfields.

\section{Methodology}\label{sec:methodology}

This paper was produced through a human--AI collaborative workflow using a lightweight, reasoning-oriented configuration of Co-Scientist~\cite{Gottweis2025}, with additional assistance from Gemini Deep Think~\cite{DeepMind2026} and Gemini Pro~\cite{DeepMind2026Pro}. The human authors formulated the problem, directed the exploratory prompts, checked the resulting arguments, repaired gaps, supplied missing hypotheses, and wrote the final exposition. The AI systems assisted by suggesting candidate reductions, examples, proof strategies, and draft arguments. All statements and proofs in the final paper were reviewed and validated by the human authors, who take responsibility for their correctness. We regard the paper as an example of AI-assisted mathematical discovery, not as an autonomous AI proof. Further details are given in Appendix~\ref{app:ai-assisted-discovery}.

\newpage
\appendix
\section{Numerical plots}

This appendix records numerical values of the two normalized quantities discussed in the paper.  The data were computed for cyclic groups \(\mathbb Z/n\mathbb Z\) for \(2\le n\le 3000\).  The first plot is intended only as a small-range illustration: Theorem~\ref{thm:large-ratio-positive-density} and Corollary~\ref{cor:ratio-asymptotics} show that \(\Ccal(\Z/n\Z)/\varphi(n)\) is unbounded, although this behavior is not visible in such a short numerical range.\\\

\begin{figure}[htbp]
\centering
\begin{tikzpicture}
\begin{axis}[
    width=0.92\textwidth,
    height=0.46\textwidth,
    xlabel={$n$},
    ylabel={$\mathcal C(\mathbb Z/n\mathbb Z)/\varphi(n)$},
    grid=both,
    mark size=0.45pt,
    xmin=2,
    xmax=3000,
    ymin=0.95,
    ymax=2.05,
]
\addplot+[only marks] table[x index=0,y index=1] {TableChowlaOverPhi.dat};
\end{axis}
\end{tikzpicture}
\caption{Numerical values of $\mathcal C(\mathbb Z/n\mathbb Z)/\varphi(n)$ for $2\le n\le 3000$.}
\label{fig:chowla-over-phi}
\end{figure}
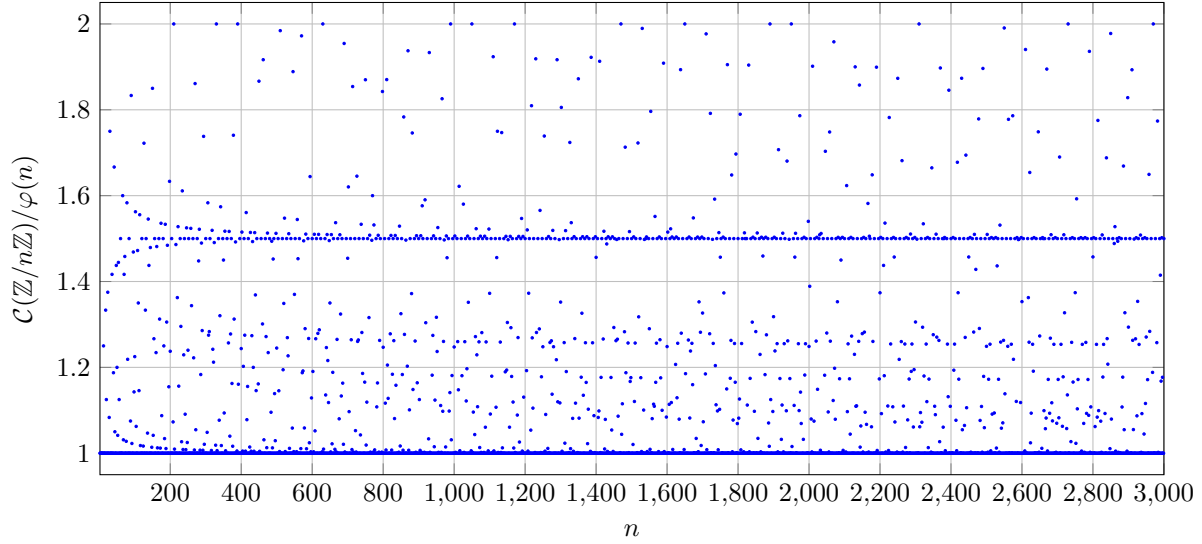

\begin{figure}[htbp]
\centering
\begin{tikzpicture}
\begin{axis}[
    width=0.92\textwidth,
    height=0.46\textwidth,
    xlabel={$n$},
    ylabel={$\mathcal C(\mathbb Z/n\mathbb Z)/n$},
    grid=both,
    mark size=0.45pt,
    xmin=2,
    xmax=3000,
    ymin=0,
    ymax=1.05,
]
\addplot+[only marks] table[x index=0,y index=1] {TableChowlaOverN.dat};
\end{axis}
\end{tikzpicture}
\caption{Numerical values of $\mathcal C(\mathbb Z/n\mathbb Z)/n$ for $2\le n\le 3000$.}
\label{fig:chowla-over-n}
\end{figure}
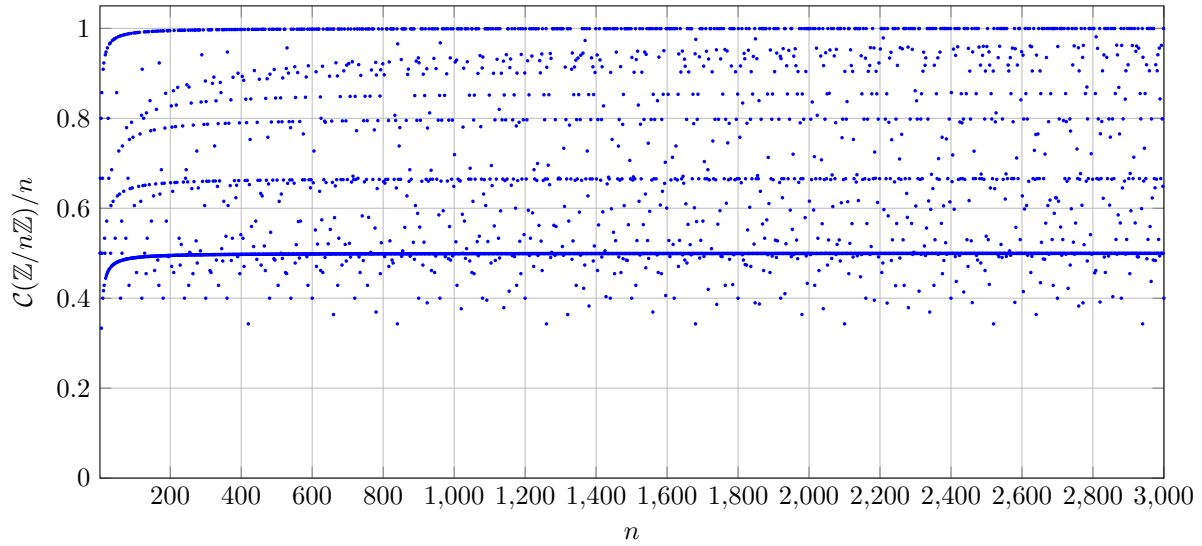

\medskip

\clearpage
\section{AI-assisted discovery using a lightweight reasoning configuration of Co-Scientist}
\label{app:ai-assisted-discovery}

This appendix documents the AI-assisted component of the discovery process. It is intended as a rigorous methodological record of the human--AI collaboration rather than a benchmark or model-performance claim.

\subsection{Overview}

In this appendix we describe the role of a lightweight, reasoning-oriented configuration of Co-Scientist~\cite{Gottweis2025}, backed by Gemini Deep Think~\cite{DeepMind2026} and Gemini Pro~\cite{DeepMind2026Pro}, in discovering the mathematical results presented in the main paper. While broader, multi-agent frameworks have been developed to support fully autonomous mathematical research and stateful collaborative workbenches, such as Aletheia~\cite{Feng2026} and AI Co-Mathematician~\cite{Zheng2026}, the present case study focuses exclusively on deploying a compact reasoning loop to assist an expert in step-by-step mathematical deduction.

As defined in the main text, a Chowla set is a nonempty subset \(S\) of a finite group \(G\) satisfying \(\ord(x)>|S|\) for every \(x\in S\), and the extremal invariant \(\Ccal(G)\) is the maximum size of such a set. This problem---isolating Hamidoune's order condition~\cite{Hamidoune1996} as a foundational subject in its own right---had not been extensively studied in prior literature, making it a suitable testbed for the AI system's deductive capabilities; see Section~\ref{app:problem-selection}.

The results are the product of a tight, expert-guided human--AI collaboration:
\begin{itemize}[leftmargin=2em]
\item The AI system generated substantial proof-grade material, including the order-distribution reduction for \(\Ccal(G)\), fixed-cardinality counting formulas, exact formulas for cyclic and finite abelian groups, and the conceptual framework for the linear analogue of Chowla subspaces.
\item The human mathematician selected the problem, decomposed it into tractable subproblems, evaluated outputs for correctness, supplied missing hypotheses, completed analytical arguments, and authored the final manuscript.
\end{itemize}

\subsection{Reasoning-oriented configuration}

The reasoning-oriented configuration used in this work is a streamlined variant of Co-Scientist~\cite{Gottweis2025}. The full system supports long-horizon autonomous exploration---including literature retrieval, knowledge-base construction, focus-area generation, and multi-round hypothesis evolution---capabilities that can be effective for open-ended empirical domains but introduce latency and compounding-error risks in verifiable mathematics, where correctness can and should be checked at each step. This configuration therefore disables auxiliary generation pathways, such as self-play, article exploration, and deep research, and retains only the core loop: generation, review, pairwise ranking, and iterative improvement. The present appendix accordingly reports a case study in which this targeted configuration contributed substantial proof-grade material under continuous expert guidance, not an autonomous discovery by the full Co-Scientist system.

The pipeline executes in two rounds with the following stages.

\medskip
\noindent\textbf{Round 0 (Initial Generation).}
\begin{enumerate}[leftmargin=2.6em,label=\textup{(\arabic*)}]
\item \textbf{Parallel Generation.} \(N\) independent solution attempts are generated in parallel, each via a separate Deep Think invocation that performs multi-step reasoning with tool use, including code execution and search.
\item \textbf{Review.} Each candidate is independently reviewed via prompt-driven critique that identifies strengths, weaknesses, and potential logical gaps.
\item \textbf{Placement Tournament.} Ideas are compared head-to-head in pairwise matches. Each match presents two candidates, together with their reviews, to the model, which selects the stronger one. Match order is randomized and swapped to mitigate position bias. Elo ratings are updated after each match to produce an initial ranking.
\item \textbf{Deep Verification.} All \(N\) ideas undergo multi-aspect verification: factual claims are extracted and checked, underlying assumptions are identified and tested, and internal logical coherence is evaluated. This produces a structured verification summary for each idea.
\item \textbf{Final Tournament.} A second tournament is run using only the deeply verified ideas, incorporating the verification results into the ranking. Low-ranking ideas are pruned according to Elo thresholds.
\end{enumerate}

\noindent\textbf{Round 1 (Improvement).}
\begin{enumerate}[leftmargin=2.6em,label=\textup{(\arabic*)},start=6]
\item \textbf{Iterative Refinement.} The top-ranked ideas are fed back to Deep Think along with their review and verification feedback. Deep Think generates \(N'\le N\) improved variants that attempt to patch identified flaws and extend the reasoning.
\item \textbf{Review, Verification, and Tournament.} The improved ideas undergo the same review, deep verification, and pairwise ranking process, now competing against both the original and improved candidates. Low-ranking ideas are again pruned according to Elo thresholds.
\end{enumerate}

\noindent\textbf{Final Output.}
\begin{enumerate}[leftmargin=2.6em,label=\textup{(\arabic*)},start=8]
\item \textbf{Overview Report.} The surviving top-ranked ideas are assembled into a structured report including complete proof chains and derivations. Reasoning-preserving prompts instruct the summarizer to maintain full logical chains rather than abstracting them into high-level summaries.
\end{enumerate}

In our experiments, the generation and refinement stages, steps 1 and 6, were driven by Gemini Deep Think, specifically Pro with extended thinking, while the review, verification, and tournament stages, steps 2--5 and 7, used Gemini Pro. The overview report, step 8, used Gemini Flash with reasoning-preservation instructions. Crucially, the review, ranking, and pruning components function as heuristic filters to optimize search efficiency and reduce cognitive load; they are not formal symbolic proof checkers. Final responsibility for establishing mathematical soundness remained with the human mathematician.

\begin{figure}[htbp]
\centering
\begin{tikzpicture}[
  x=1cm,y=1cm,
  box/.style={draw,rounded corners=1.5pt,align=center,minimum width=2.25cm,minimum height=0.8cm,font=\scriptsize,inner sep=2pt},
  deep/.style={box,fill=blue!13},
  pro/.style={box,fill=red!13},
  flash/.style={box,fill=yellow!20},
  in-out/.style={box,fill=gray!20},
  arr/.style={-{Latex[length=2mm]},thick}
]
\node[font=\scriptsize\bfseries] at (-1.8,1.65) {Round 0};
\node[in-out] (goal) at (0.5,1.65) {Research\\Goal};
\node[deep] (gen) at (3.4,1.65) {Generate\\ \(N\) Ideas\\ Deep Think};
\node[pro] (rev0) at (6.3,1.65) {Review\\ Pro};
\node[pro] (elo0) at (9.2,1.65) {Elo\\Tournament\\ Pro};
\node[pro] (ver0) at (12.1,1.65) {Deep\\Verification\\ Pro};
\draw[arr] (goal)--(gen); \draw[arr] (gen)--(rev0); \draw[arr] (rev0)--(elo0); \draw[arr] (elo0)--(ver0);

\node[font=\scriptsize\bfseries] at (-1.8,-0.1) {Round 1};
\node[pro] (ver1) at (3.4,-0.1) {Deep\\Verification\\ Pro};
\node[pro] (elo1) at (6.3,-0.1) {Elo\\Tournament\\ Pro};
\node[pro] (rev1) at (9.2,-0.1) {Review\\ Pro};
\node[deep] (imp) at (12.1,-0.1) {Improve\\\(N'\) Ideas\\ Deep Think};
\draw[arr] (ver0)--(imp); \draw[arr] (imp)--(rev1); \draw[arr] (rev1)--(elo1); \draw[arr] (elo1)--(ver1);

\node[font=\scriptsize\bfseries] at (-1.8,-1.85) {Summary};
\node[flash] (over) at (3.4,-1.85) {Overview\\Report\\ Flash};
\node[in-out] (final) at (6.3,-1.85) {Final\\Output};
\draw[arr] (ver1)--(over); \draw[arr] (over)--(final);

\node[deep,minimum width=1.65cm] at (9.2,-1.85) {Deep Think};
\node[pro,minimum width=1.25cm] at (11,-1.85) {Pro};
\node[flash,minimum width=1.25cm] at (12.6,-1.85) {Flash};
\end{tikzpicture}
\caption{Co-Scientist's reasoning-configuration pipeline. Round 0 generates three independent solutions via Deep Think, reviews them, ranks them through a pairwise Elo tournament, and performs multi-aspect deep verification, including claim extraction, factuality checking, assumption analysis, and coherence evaluation. Round 1 improves top-ranked ideas, with \(N'<N\), using review and verification feedback, then repeats the same evaluation cycle. The overview report preserves full reasoning chains through reasoning-preserving prompts.}
\label{fig:ai-pipeline}
\end{figure}
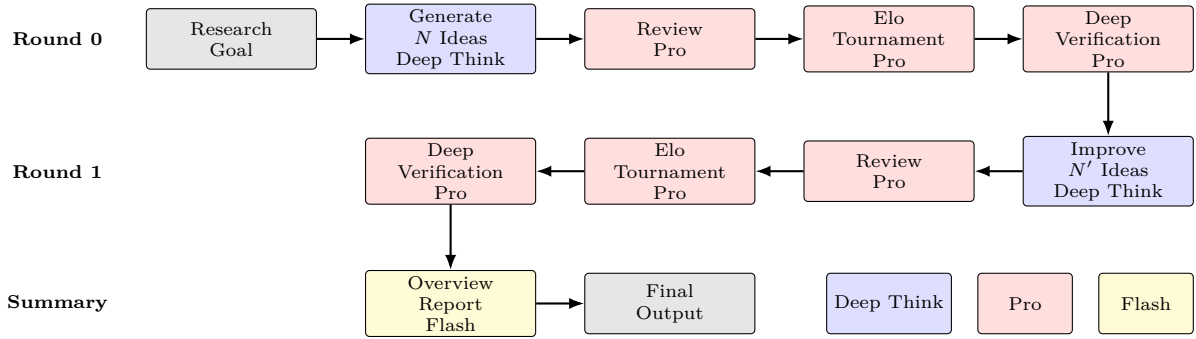

\subsection{Problem-selection rationale and literature status}
\label{app:problem-selection}

The mathematical problem treated in this paper was deliberately selected to stress abstract mathematical reasoning and rule out direct literature retrieval or training-data memorization. A persistent challenge in evaluating frontier AI systems on mathematical tasks is data contamination: many well-known or benchmark problems possess established proof templates in the public domain, allowing a model to synthesize solutions by applying standard proof templates rather than genuine novel deduction. To mitigate this confound, we chose a problem setting that was intentionally absent from the mathematical literature.

The foundational motivation stems from Chowla-type order conditions in additive combinatorics. While the core philosophy originates with S. Chowla, it was specifically Y.~O. Hamidoune~\cite{Hamidoune1996} who introduced hypotheses requiring group-element orders to be strictly larger than the size of the relevant subset in order to eliminate small periodic obstruction loops in product-growth theorems. The present work isolates this constraint, transforming it into an independent, open-ended study of the maximum size and exact enumeration of these subsets.

To verify that the resulting invariants, \(\Ccal(G)\) and \(\Ccal(L/K)\), were unstudied, the first author conducted a comprehensive literature baseline check before the AI interaction. The structural status of the problem is summarized as follows.
\begin{itemize}[leftmargin=2em]
\item \textbf{The linear-analogue origin.} The conceptual translation of this constraint into linear algebra---the definition of a Chowla subspace---was first introduced by the first author in 2025~\cite{Aliabadi2025}. That initial work did not establish exact formulas or investigate the extremal invariant \(\Ccal(L/K)\).
\item \textbf{Community verification through MathOverflow.} To test whether the extremal and enumerative problems for both groups and fields had known solutions, the first author posted a public inquiry on September 23, 2025, entitled ``Size of Chowla sets and Chowla subspaces''~\cite{MOChowla}. The inquiry remained unanswered, providing an additional indication that no standard formula or established approach was readily available for these specific questions.
\end{itemize}

During the interaction, the AI system had web-search capabilities, and the run logs show that it used them. The system retrieved the MathOverflow and LiveJournal mirror posts authored by the human mathematician, together with foundational work by Hamidoune and related literature on group matchability. Because the retrieved posts contained no solutions and the formal literature supplied only the foundational definitions and background, the searches did not return the theorem-level conclusions established here. The system still had to perform mathematical reasoning to propose such ingredients as the order-distribution reduction and the subspace inclusion--exclusion architecture.

We emphasize that this selection criterion does not constitute a formal claim of ``pure reasoning,'' since the system undoubtedly benefited from standard background mathematics acquired during pre-training. The final proofs in the main text are self-contained and rely on classical graduate-level tools, including finite group theory and Galois theory, finite-field subfield lattices, Euler's totient function, Gaussian binomial coefficients, and M\"obius inversion on divisor and subspace lattices. The system applied these standard tools to an extremal invariant for which no existing solution was supplied by the retrieved sources. The appropriate claim is therefore one of collaborative human--AI mathematical discovery, not proof of independence from all learned mathematical patterns.

\subsection{Interaction protocol and expert prompting}

The AI-assisted portion began with two broad prompts concerning the extremal and enumerative behavior of Chowla sets. The first asked for upper and lower bounds, and possible exact formulas, for \(\Ccal(G)\). The second asked for the number of Chowla subsets of fixed cardinality. These broad prompts produced several useful ideas, including the reduction to order-survival functions and cyclic-group formulas.

\begin{tcolorbox}[colback=white,colframe=black!65,title=Prompt 1: Bounds for \(\Ccal(G)\),fonttitle=\bfseries,breakable]
Let \(G\) be a finite group. A nonempty subset \(S\subseteq G\) is called a Chowla subset if every element of \(S\) has order strictly larger than \(|S|\); that is,
\[
 \ord(x)>|S|\qquad\text{for every }x\in S.
\]
Chowla subsets were introduced and studied by Y.~O. Hamidoune.

For a finite group \(G\), let \(\Ccal(G)\) denote the maximum size of a Chowla subset of \(G\). For example, if \(p\) is prime, then \(\Ccal(\Z/p\Z)=p-1\).

\textbf{Question.} What are some nontrivial upper and lower bounds for \(\Ccal(G)\) for a given finite group \(G\)?

\textbf{Remark.} A natural starting point is \(\Ccal(\Z/n\Z)\) for finite cyclic groups. A straightforward lower bound is Euler's totient function:
\[
 \varphi(n)=n\prod_{p\mid n}\left(1-\frac1p\right).
\]
\end{tcolorbox}

\begin{tcolorbox}[colback=white,colframe=black!65,title=Prompt 2: Counting Chowla subsets of fixed cardinality,fonttitle=\bfseries,breakable]
Let \(G\) be a finite group. A nonempty subset \(S\subseteq G\) is called a Chowla subset if every element of \(S\) has order strictly larger than \(|S|\); that is,
\[
 \ord(x)>|S|\qquad\text{for every }x\in S.
\]
Chowla subsets were introduced and studied by Y.~O. Hamidoune.

\textbf{Question.} For each integer \(n\) with \(1<n<|G|\), determine or estimate the number of Chowla subsets \(S\subseteq G\) of cardinality \(n\).

\textbf{Remark.} A natural starting point is \(G=\Z/n\Z\), where the condition \(\ord(x)>|S|\) can be analyzed explicitly in terms of divisors of \(n\).
\end{tcolorbox}

After reviewing the first outputs, the human mathematician decomposed the problem into a sequence of more focused mathematical subquestions. These included cyclic groups, threshold-form constructions, equality with Euler's totient function, finite abelian groups, abelian \(p\)-groups, the linear analogue for field extensions, and finite-field subspace-counting problems. Each refined prompt supplied definitions, notation, and a specific mathematical target. Table~\ref{tab:prompt-provenance} summarizes the interaction structure at a high level.

\begin{longtable}{>{\raggedright\arraybackslash}p{0.15\textwidth}>{\raggedright\arraybackslash}p{0.22\textwidth}>{\raggedright\arraybackslash}p{0.27\textwidth}>{\raggedright\arraybackslash}p{0.20\textwidth}}
\caption{High-level prompt provenance.}\label{tab:prompt-provenance}\\
\toprule
\textbf{Prompt class} & \textbf{Mathematical target} & \textbf{Useful output} & \textbf{Human role}\\
\midrule
\endfirsthead
\toprule
\textbf{Prompt class} & \textbf{Mathematical target} & \textbf{Useful output} & \textbf{Human role}\\
\midrule
\endhead
Broad prompt 1 & Bounds and formulas for \(\Ccal(G)\) & Order-survival function, threshold characterization, and cyclic examples & Supplied the definition, motivation, and Hamidoune context.\\
\addlinespace
Broad prompt 2 & Counting Chowla subsets of fixed cardinality & Reduction of enumeration to choosing subsets from the elements of order greater than \(m\) & Supplied the fixed-cardinality question and cyclic starting point.\\
\addlinespace
Focused finite-group prompt & Determine \(\Ccal(G)\) from the order distribution & Precise theorem-level reduction using \(N_G(m)\) & Isolated the correct abstraction.\\
\addlinespace
Focused cyclic prompts & Compute \(\Ccal(\Z/n\Z)\), threshold form, and equality with \(\varphi(n)\) & Divisor formulas, threshold construction, equality criterion, and explicit cases & Broke the cyclic problem into tractable subquestions.\\
\addlinespace
Focused finite-abelian prompt & Compute exact order counts in finite abelian groups & \(T_G(d)=\prod_i\gcd(d,n_i)\), M\"obius inversion, and \(p\)-group consequences & Selected the invariant-factor framework.\\
\addlinespace
Focused linear-analogue prompt & Define and bound Chowla subspaces in field extensions & Translation to \([K(a):K]>\dim_KA\) and the \(d_{\max}\) obstruction & Proposed the analogy and later corrected its scope.\\
\addlinespace
Focused finite-field prompts & Count Chowla subspaces through subspace avoidance & Gaussian-binomial formulas and inclusion--exclusion structure & Selected the chain case and first non-chain case.\\
\bottomrule
\end{longtable}

The strongest outputs were produced when the human mathematician identified the right intermediate problem, formulated it with precise notation, and asked for a theorem-level treatment. For each focused prompt, the reasoning configuration generated candidates, reviewed and ranked them, refined promising outputs, and returned retained ideas. The human mathematician then inspected the results, rejected incorrect suggestions, and determined which directions were worth pursuing. The remaining work consisted of validating proposed arguments, filling gaps, correcting overstatements, adding missing hypotheses, and writing the final proofs.

\subsection{Provenance of mathematical ideas}

The central claim of this appendix is not that the AI system independently produced a complete formal paper. The claim is that it produced substantial theorem-level material and several key proof ideas used in the final paper.

Because estimates such as ``approximately 80\% of the useful content came from the AI system'' are necessarily subjective, we do not treat such a percentage as a formal measurement. Instead, Table~\ref{tab:idea-provenance} reports idea-level provenance. Any qualitative percentage should be interpreted as an expert assessment rather than as a quantitative attribution metric.

\begingroup
\footnotesize
\renewcommand{\arraystretch}{1.08}
\setlength{\tabcolsep}{3pt}
\begin{longtable}{>{\raggedright\arraybackslash}p{0.18\textwidth}>{\raggedright\arraybackslash}p{0.22\textwidth}>{\raggedright\arraybackslash}p{0.26\textwidth}>{\raggedright\arraybackslash}p{0.22\textwidth}}
\caption{Provenance of mathematical ideas used in the final paper.}\label{tab:idea-provenance}\\
\toprule
\textbf{Final-paper component} & \textbf{AI output} & \textbf{Role in the final paper} & \textbf{Human contribution}\\
\midrule
\endfirsthead
\toprule
\textbf{Final-paper component} & \textbf{AI output} & \textbf{Role in the final paper} & \textbf{Human contribution}\\
\midrule
\endhead
Definition of Chowla set and \(\Ccal(G)\) & Initial human prompt; AI formalization in broad and focused runs & Sets up the invariant studied throughout the paper & Selected the object, connected it to Chowla-type order conditions, and fixed notation.\\
\addlinespace
Order-distribution reduction & Focused finite-group run deriving \(\Ccal(G)=\max\{m:N_G(m)\ge m\}\) & Becomes Proposition~\ref{prop:order-distribution}, the main structural reduction for finite groups & Checked the proof, simplified the exposition, and aligned notation.\\
\addlinespace
Counting formula for fixed cardinality & Broad enumeration prompt and focused finite-group run & Gives the formula \(\binom{N_G(m)}m\) for the number of Chowla sets of size \(m\) & Verified the exact statement and integrated it with Proposition~\ref{prop:order-distribution}.\\
\addlinespace
Cyclic divisor formula & Focused cyclic-group run using the element-order distribution in \(\Z/n\Z\) & Becomes the formula involving \(\sum_{d\mid n,\,d>m}\varphi(d)\) & Polished the proof and connected it to later arithmetic consequences.\\
\addlinespace
Threshold-form maximum sets & Focused cyclic-group run selecting elements of largest order first & Supports the threshold interpretation of maximum Chowla sets & Adapted the statement to the final notation and precise scope.\\
\addlinespace
Criterion for \(\Ccal(\Z/n\Z)=\varphi(n)\) & Focused cyclic-group equality run & Leads to the criterion involving the least prime divisor \(p_1\) and \(\varphi(n)\ge n/p_1-1\) & Corrected, streamlined, and stated the result in final form.\\
\addlinespace
Prime-power and two-prime corollaries & Focused cyclic-group run & Provides explicit families where equality or an improved bound can be determined & Checked cases, repaired arguments where needed, and organized them as corollaries.\\
\addlinespace
Asymptotic behavior of \(\Ccal(\Z/n\Z)/\varphi(n)\) & AI outputs suggested arithmetic directions and threshold sets & The final paper proves liminf \(1\), infinite limsup, and positive lower-density threshold results & Supplied or substantially repaired the analytic number-theoretic arguments and final rigor.\\
\addlinespace
Finite abelian group order counts & Focused finite-abelian-group run deriving \(T_G(d)=\prod_i\gcd(d,n_i)\) and M\"obius inversion for \(a_G(d)\) & Becomes the finite abelian group formula for \(\Ccal(G)\) & Verified the invariant-factor setup and incorporated standard references.\\
\addlinespace
Abelian \(p\)-group closed form & Focused finite-abelian-group run & Supports the closed form for cyclic and noncyclic abelian \(p\)-groups & Completed and checked the case analysis.\\
\addlinespace
Definition of Chowla subspace & Focused linear-analogue run translating \(|S|\) to \(\dim_KA\) and \(\ord(x)\) to \([K(a):K]\) & Introduces the linear analogue in finite extensions & Judged the analogy mathematically meaningful and fixed the final definition.\\
\addlinespace
Upper bound using \(d_{\max}(L/K)\) & Focused linear-analogue run identifying largest proper intermediate fields as obstructions & Becomes the upper bound \(\Ccal(L/K)\le n-d_{\max}(L/K)\) & Corrected the scope and supplied a rigorous dimension-intersection proof.\\
\addlinespace
Sharpness over infinite base fields & Focused linear-analogue run suggested avoidance of proper intermediate fields & The final paper proves sharpness for finite separable extensions over infinite base fields & Added necessary hypotheses, finite-union arguments, and the final proof.\\
\addlinespace
Finite-field chain counting & Focused finite-field run on subspaces avoiding a fixed subspace & Leads to Gaussian-binomial formulas for chain cases & Selected the final formulation and aligned it with the subfield lattice.\\
\addlinespace
First non-chain finite-field case \(n=pq\) & Focused finite-field run developed subspace-avoidance and inclusion--exclusion ideas & The final paper treats two incomparable proper subfields and gives an explicit formula & Supplied the final construction, corrected details, and completed the inclusion--exclusion argument.\\
\addlinespace
Numerical plots & Not AI-generated & Illustrate the behavior of normalized ratios for cyclic groups & Human and computational contribution.\\
\bottomrule
\end{longtable}
\endgroup

This table records which portions of the final argument were proposed by the system, which were repaired by the human mathematician, and which were supplied independently by the human mathematician.

\clearpage
\subsection{Selected interaction logs: from broad exploration to focused deduction}

To illustrate the collaborative workflow, we provide selected excerpts from the interaction logs. The transcripts have been edited for length to remove routine conversational filler while preserving the explicit mathematical prompts and the theorem-level outputs generated by the AI.

\subsubsection*{1. The extremal invariant and order distribution}

\noindent\textbf{The human intent.} The exploration began with a broad inquiry into the extremal properties of the Chowla-set invariant \(\Ccal(G)\). The human mathematician provided the foundational definition and nudged the system to consider cyclic groups and Euler's totient function as a natural starting point. The goal was to see whether the system could identify a unifying structural parameter governing \(\Ccal(G)\).

\noindent\textbf{The AI contribution.} The system identified that \(\Ccal(G)\) is governed by the group's element-order distribution. It introduced an ``order-survival function'' and established an exact combinatorial characterization.

\begin{tcolorbox}[colback=gray!4,colframe=black!55,breakable]
\textbf{Human prompt.} What are some nontrivial upper and lower bounds for \(\Ccal(G)\) for a given finite group \(G\)? A natural starting point is \(\Ccal(\Z/n\Z)\) for finite cyclic groups. A straightforward lower bound is Euler's totient function.

\medskip
\textbf{AI output (excerpt).}

\textbf{Definition.} Let \(G\) be a finite group. Define the order-survival function \(N_{>k}(G)\) for an integer \(k\ge1\) by
\[
 N_{>k}(G):=|\{g\in G:\ord(g)>k\}|.
\]

\textbf{Lemma [Exact characterization of \(\Ccal(G)\)].} For any finite group \(G\) with \(|G|\ge2\), the maximum size of a Chowla subset is
\[
 \Ccal(G)=\max\{k\in\Z_{\ge1}:N_{>k}(G)\ge k\}.
\]
\end{tcolorbox}

\noindent\textbf{The human synthesis.} The human mathematician verified the proof and recognized this reduction as the correct abstraction for the paper. The notation \(N_{>k}(G)\) was refined to \(N_G(m)\), together with \(\Omega_G(m)\), producing Proposition~\ref{prop:order-distribution}. The system also proposed an exact divisor formula for cyclic groups, which the human authors polished into Theorem~\ref{thm:cyclic-exact}.

\subsubsection*{2. Fixed-cardinality enumeration}

\noindent\textbf{The human intent.} After the maximum-size characterization, the human mathematician asked for the exact number of Chowla subsets of a prescribed size.

\noindent\textbf{The AI contribution.} Building on the previous reduction, the system reduced the enumeration problem to selecting elements from the strict survival set.

\begin{tcolorbox}[colback=gray!4,colframe=black!55,breakable]
\textbf{Human prompt.} For each integer \(n\) with \(1<n<|G|\), determine or estimate the number of Chowla subsets \(S\subseteq G\) of cardinality \(n\).

\medskip
\textbf{AI output (excerpt).}

\textbf{Theorem [Exact enumeration formula].} Let \(G\) be a finite group and let \(n\) satisfy \(1<n<|G|\). If \(\omega_{>n}(G)=|\{g\in G:\ord(g)>n\}|\), then the number \(c_n(G)\) of Chowla subsets of cardinality \(n\) is
\[
 c_n(G)=\binom{\omega_{>n}(G)}{n}.
\]

\emph{Proof excerpt.} Constructing a Chowla subset \(S\) of cardinality \(n\) is mathematically identical to choosing exactly \(n\) distinct elements from the strict survival set \(\Omega_{>n}(G)\).
\end{tcolorbox}

\noindent\textbf{The human synthesis.} The human mathematician confirmed the subset-selection argument and integrated the formula with the order-distribution reduction, completing the counting statement in Proposition~\ref{prop:order-distribution}.

\subsubsection*{3. Isolating the equality case for cyclic groups}

\noindent\textbf{The human intent.} Once \(\varphi(n)\) had been identified as a universal lower bound for \(\Ccal(\Z/n\Z)\), the human mathematician asked precisely which integers achieve equality.

\noindent\textbf{The AI contribution.} The system identified the largest proper divisor, hence the smallest prime factor, as the governing obstruction and derived the boundary condition.

\begin{tcolorbox}[colback=gray!4,colframe=black!55,breakable]
\textbf{Human prompt.} For which integers \(n\ge2\) does the equality \(\Ccal(\Z/n\Z)=\varphi(n)\) hold?

\medskip
\textbf{AI output (excerpt).}

\textbf{Theorem.} Let \(n\ge2\), and let \(p\) be its smallest prime factor. Then \(\Ccal(\Z/n\Z)=\varphi(n)\) if and only if
\[
 \frac np\le \varphi(n)+1.
\]

\emph{Proof excerpt.} The remaining internal sum vanishes exactly when \(n\) has no proper divisor strictly greater than \(\varphi(n)+1\).
\end{tcolorbox}

\noindent\textbf{The human synthesis.} Because \(\varphi(n)\) and \(n/p\) are integers, the condition was streamlined to \(\varphi(n)\ge n/p_1-1\), yielding Proposition~\ref{prop:cyclic-phi-criterion}. The human authors then verified the explicit prime-power and two-prime families in the subsequent corollaries.

\subsubsection*{4. Formulating the linear analogue in field extensions}

\noindent\textbf{The human intent.} After developing the group-theoretic invariant, the human mathematician proposed translating subset cardinality and element order into the language of finite field extensions and vector subspaces.

\noindent\textbf{The AI contribution.} The system constructed a group-to-field dictionary: set cardinality became \(K\)-dimension, while element order became the degree of the generated simple subextension. It formulated the Chowla-subspace condition and identified \(d_{\max}(L/K)\) as the governing obstruction.

\begin{tcolorbox}[colback=gray!4,colframe=black!55,breakable]
\textbf{Human prompt.} What is the appropriate linear analogue of a Chowla set for a finite field extension \(K\subseteq L\)? How large can a Chowla \(K\)-subspace of \(L\) be in terms of the largest proper intermediate field of \(L/K\)?

\medskip
\textbf{AI output (excerpt).}

\textbf{Definition [Chowla subspace].} Let \(K\subseteq L\) be a finite field extension of degree \(n>1\). A nonzero \(K\)-subspace \(A\subseteq L\) is a Chowla subspace if
\[
 [K(a):K]>\dim_KA\qquad\text{for every }0\ne a\in A.
\]

\textbf{Theorem [Universal upper bound].} For any finite field extension \(L/K\),
\[
 \Ccal(L/K)\le n-d_{\max}(L/K).
\]
\end{tcolorbox}

\noindent\textbf{The human synthesis.} The human mathematician judged the structural translation meaningful, verified the geometric dimension-intersection proof, corrected the scope of claims concerning infinite and finite base fields, and wrote the rigorous exposition in Section~\ref{sec:chowla_subspaces_in_finite_extensions}; the upper bound appears as Theorem~\ref{thm:linear-upper-bound}.

\clearpage
\subsection{Limitations}

This appendix reports a single case study, not a benchmark. Several limitations should be emphasized:
\begin{itemize}[leftmargin=2em]
\item The problem was selected and decomposed by a human expert whose mathematical judgment substantially shaped the system's success. The results should not be interpreted as evidence that non-expert users can obtain comparable outputs from vague prompts.
\item The review and ranking stages were model-based heuristic filters, not formal verification procedures. A candidate favored by the system could still contain serious mathematical errors.
\item The boundary between AI-generated content, human repair, and standard mathematical background is not always sharp. Although Tables~\ref{tab:prompt-provenance} and~\ref{tab:idea-provenance} provide concrete provenance records, final attribution necessarily reflects human interpretation of the logs.
\item The problem was chosen to reduce the likelihood of direct literature recombination, but the model may still have drawn on general proof patterns learned during training.
\end{itemize}

The appropriate conclusion is therefore limited but meaningful: in this case study, a compact reasoning configuration of Co-Scientist, guided by expert mathematical prompting, produced substantial proof-grade material for a research-level paper on Chowla sets and subspaces.

\end{document}